\documentclass[a4paper,draft]{amsart}
\usepackage{amsthm,enumerate,
verbatim,amssymb,color}

\def\Cap{\operatorname{Cap}}
\def\Cp{C_*}

\def\E{\mathcal E}

\def\H{\mathit{co}\mathcal H}

\def\K{\mathcal K}

\def\diam{\operatorname{diam}}
\def\dist{\operatorname{dist}}

\def\er{\mathbb R}

\def\rn{\mathbb R^n}

\def\haus{\mathcal H^{n{-}1}}

\def\er{\mathbb R}
\def\Lip{\operatorname{Lip}}
\def\cp{\operatorname{cap}}
\def\Cp{\operatorname{Cap}}
\def\loc{\mathrm{loc}}

\def\rn{\mathbb R^n}

\def\inf{\operatorname{inf}}
\def\sup{\operatorname{sup}}

\newdimen\vintkern
\newdimen\vintkern
\vintkern10pt\usepackage{amsmath}
\usepackage{amsfonts}
\usepackage{amssymb}

\def\vint{{-}\kern-\vintkern\int}

\def\eqn#1$$#2$${\begin{equation}\label#1#2\end{equation}}
\def\inte{\operatorname{int}_*\!}
\def\exte{\operatorname{ext}_*\!}

\def\standhaus{\widetilde{\mathcal H}^{n{-}1}}

\def\H{\mathcal H}

\newtheorem{thm}{Theorem}[section]
\newtheorem{lemma}[thm]{Lemma}
\newtheorem{cor}[thm]{Corollary}
\newtheorem{prop}[thm]{Proposition}

\theoremstyle{definition}

\newtheorem{examples}[thm]{Examples}

\newtheorem{rmrk}[thm]{Remark}

\long\def\colred#1\endred{{\color{red}#1}}
\long\def\colgreen#1\endgreen{{\color{green}#1}}
\long\def\colblue#1\endblue{{\color{blue}#1}}

\title[$AM$--modulus and and Hausdorff measure of codimension one]{\vspace{0.5in}
{\bf $AM$--modulus and Hausdorff measure of codimension one in metric measure spaces\\}}

\author{Vendula Honzlov\'a Exnerov\'a}
\address{Department of Mathematical Analysis, Faculty of Mathematics and Physics,
Charles University in Prague, So\-ko\-lovsk\'a 83,  Prague 8, 186\,75 Czech Republic}
\email{honzlova@karlin.mff.cuni.cz}
\author{Jan Mal\'y}
\address{Department of Mathematical Analysis, Faculty of Mathematics and Physics,
Charles University in Prague, So\-ko\-lovsk\'a 83,  Prague 8, 186\,75 Czech Republic}
\email{maly@karlin.mff.cuni.cz}

\author{Olli Martio}
\address{Department of Mathematics and Statistics, FI-00014 University of Helsinki, Finland}
\email{olli.martio@helsinki.fi}

\thanks{The first and the second authors have been supported by the grant GA\,\v{C}R 
P201/18-07996S of the Czech Science Foundation.}

\begin{document}

\begin{abstract}
Let $\Gamma(E)$ be the family of all paths which
meet a set $E$ in the metric measure space $X$. The set function $E \mapsto AM(\Gamma(E))$ defines the $AM$--modulus measure in $X$ where $AM$ refers to the approximation modulus \cite{M}. We 
  compare $AM(\Gamma(E))$ to the Hausdorff measure $co\H^1(E)$ of codimension one in $X$ and show that
$$co\H^1(E) \approx AM(\Gamma(E))$$
 for Suslin sets $E$ in $X$. This leads to  a new characterization of sets of finite perimeter in $X$ in terms of the $AM$--modulus. We
 also study the level sets of $BV$ functions and show that for a.e. $t$ these sets
 have finite $co\H^1$--measure.
Most of the results are new also in $\rn$.
\end{abstract}


\maketitle

\section{Introduction}

In a metric measure space $X$ the modulus of a curve
family offers a substitute for the Fubini theorem
and provides an important tool for analysis in $X$, see e.g. \cite{BB}. The $M_p$--modulus,
$p \geq 1$, is used to create a space  in $X$ similar to the
Sobolev space $W^{1,p}(\rn)$ and the $AM$--modulus
 was introduced as a weaker version than the
$M_1$--modulus to study functions of bounded variation in $X$ and in $\rn$, see \cite{M}, \cite{HEMM} and \cite{HEMM2}.

Let $\Gamma(E)$ be the family of all
paths in $X$ which meet the set $E \subset X$.
The set function $E \mapsto AM(\Gamma(E))$ defines a metric outer measure, the $AM$--modulus measure, in $X$ and satisfies
\begin{equation} \label{Eq1}
AM(\Gamma(E)) \leq C\, co\H^1(E)
\end{equation}
provided that the measure $\mu$ is doubling, see
Theorem \ref{Tupper} below. Here $co\H^1$ refers to
the Hausdorff measure of codimension one in $X$. We also present the generalization of
(\ref{Eq1}) for all measures $co\H^p$, $p \geq 1$.

In this paper we are interested in the inequalities opposite to (\ref{Eq1}). Such an 
inequality was obtained in \cite{HEMM2} for sets $E$ contained in $(n-1)$--rectifiable 
sets in $\rn$. Here we show that this inequality holds for 
Borel sets in $X$, and more generally for Suslin sets and for arbitrary 
sets with $\sigma$--finite $co\H^1$--measure, provided that $X$ satisfies standard regularity assumptions, i.e. the measure $ \mu$ in $X$ is doubling, $X$ is complete and supports the
Poincar\'e inequality.
 Thus in $\rn$ the standard $(n-1)$--Hausdorff measure $\H^{n-1}$ satisfies
\begin{equation} \label{EHRn}
 \H^{n-1}(E) \approx AM(\Gamma(E))
 \end{equation}
for all Suslin sets and arbitrary sets of $\sigma$--finite $\H^{n-1}$--measure. Note that the
ordinary $M_p$--modulus is more adapted to measure
the family $\Gamma(E, \Omega)$ of all curves which join $E$ to the complement of a fixed open set $\Omega$ and then the result corresponds to the
$p$--capacity of $E$. Thus the relation to
the $(n-p)$--dimensional Hausdorff measure is mediated through the capacity and does not provide as close 
a connection as (\ref{EHRn}), see also Remark \ref{RAMRn}.

We apply the above results to study the $AM$--modulus of path families which are closely associated with sets of finite perimeter in $X$. 
Although there is extensive literature on sets of finite perimeter in metric measure spaces, see \cite{A2}, \cite{AFP}, \cite{KKST},
\cite{KL}, \cite{La} and \cite{Mi}, the $AM$--modulus has not yet been used to characterize sets of finite perimeter in $X$ and our results extend the characterizations obtained in \cite{HEMM2} 
in $\rn$ to $X$. 

We study  the level sets of a $BV$ function $u$ in the final section and show that these sets
have finite $co\H^1$--measure for a.e. $t$. In particular, it follows that the ordinary level 
set $u^{-1}(t)$ of a continuous $BV$ function 
$u$  has finite $co\H^1$--measure for a.e. $t$.

\section{Preliminaries}

Let $(X,d)$ be a metric space and $\mu$ a Borel regular measure in $X$. The measure 
$\mu$ is \textit{doubling} if there is a constant $C_{\mu}$ such that 
$\mu(B(x,2r)) \leq C_{\mu}\, \mu(B(x,r))$ and $0 <\mu(B(x,r))<\infty$ for 
all open balls $B(x,r)$ in $X$.

A continuous mapping $\gamma\colon [a,b]\to X$ is called a \textit{curve}.
We say that a curve $\gamma$ is a \textit{path}
if it has a finite and non--zero total length; in this case we parametrize
$\gamma$ by its arclength. 
The \textit{locus} of $\gamma$ is defined as $\gamma([0,\ell])$ and 
denoted by $\langle\gamma\rangle$.

We refer to \cite{M} and \cite{HEMM} for the properties of the $AM_p$--modulus and to 
\cite{BB} and \cite{Fug} for those of the $M_p$--modulus. For completeness we recall the definitions.

Let $\Gamma$ be a family of paths in $X$. 
A non--negative Borel function $\rho$ is $M$--\textit{admissible}, or simply admissible, for $\Gamma$ if
$$ \int_{\gamma} \rho \, ds \geq 1$$
for every $\gamma \in \Gamma$. For $p \geq 1$ the $M_p$--\textit{modulus} of $\Gamma$ is defined as
$$M_p(\Gamma) = \inf \int_X \rho^p \, d\mu$$
where the infimum is taken over all admissible functions $\rho$. 

A sequence of non--negative Borel functions $\rho_i $, $i = 1, \,2, \,...\,,$ is 
$AM$--\textit{admissible}, or simply admissible,
for $\Gamma$ if
\begin{equation}
\label{eqn1}
\liminf_{i \rightarrow \infty} \int_{\gamma} 
\rho_i \, ds \geq 1
\end{equation}
for every $\gamma \in \Gamma$. The \textit{approximation
modulus}, $AM_p$--modulus for short, of $\Gamma$ is defined as
\begin{equation}
\label{eqn2}
AM_p(\Gamma) = \inf_{(\rho_i)} \Big\{ \liminf_{i \rightarrow \infty} \int_{X} \rho_i^p \, d\mu \Big\}
\end{equation}
where the infimum is taken over all $AM$--admissible sequences $(\rho_i)$ for $\Gamma$. We mostly consider the $AM_1$--modulus and use the abbreviation $AM = AM_1$. Note that for $p> 1$, $AM_p(\Gamma) = M_p(\Gamma)$ for every path family $\Gamma$ in $X$, see \cite[Theorem 1]{HEMM}, however, sometimes it is easier to use the $AM_p$--modulus than the $M_p$--modulus.
Note also that $AM(\Gamma)\le M_1(\Gamma)$ for all path families $\Gamma$ in $X$ and
it could happen that $AM(\Gamma)= 0$ but
$M_1(\Gamma) = \infty$ for some family $\Gamma$.

We define the $AM_c$--modulus of $\Gamma$ with respect to the $AM$--modulus with the 
difference that the admissible sequence are now required to consist of 
continuous functions.

The $AM$ modulus or the $AM_c$ modulus can be also assigned to a 
family $\mathcal E$ of measures, $\int_{\gamma}\rho_i\,ds$, 
$\gamma\in\Gamma$, is then 
replaced by $\int_X\rho_i\,d\nu$, $\nu\in \mathcal E$. For a more 
precise definition we refer to \cite{HEKMM}. 

For $E \subset X$,  $\Gamma(E)$ denotes the family of all paths which meet $E$. From \cite[Theorem 1]{HEMM2} it follows that 
the set function $\phi: E \mapsto AM(\Gamma(E))$ 
is a metric outer measure in $X$ and hence all Borel sets are $\phi$ measurable. Almost the same proof shows that for $p \geq 1$ the set functions $E \mapsto AM_p(\Gamma(E))$
and  $E \mapsto M_p(\Gamma(E))$ also define metric outer measures in $X$. 

We denote by $\H^{n-p}$ the ordinary Hausdorff measure
of codimension $p$ in $\rn$. In metric spaces, the dimension $n$ is not always clearly determined. The right replacement of $\H^{n-p}$ is then the \textit{Hausdorff
measure $co\H^p(E)$ of codimension} $p$ defined as
$$co\H^p(E)=\sup_{\delta>0}\, co\H^p_{\delta}(E)$$
where
for $\delta > 0$
$$co\H^p_{\delta}(E)=\inf
\Bigl\{
\sum_{j=1}^{\infty}\frac{\mu(B(x_j,r_j))}{r_j^p}\colon E\subset
\bigcup_{j=1}^{\infty} B(x_j,r_j),\;\sup_jr_j<\delta
\Bigr\}$$
denotes the $\delta$--content associated with $ co\H^p(E)$. It is easily checked that in $\rn$, $co\H^p$ agrees with the $\H^{n-p}$--measure up to a multiplicative constant. 

In the following, we are chiefly interested in $co\H^1(E)$ and its dependence on $AM(\Gamma(E))$ and we first consider upper bounds for $AM(\Gamma(E))$ in terms of $co\H^1(E)$. Such a
result was presented in
\cite[Theorem 3.17]{M} and for completeness we include a proof. For $p>1$ we present a stronger version in $X$ and extend the implication,  see \cite[Theorem 2.27]{HKM} and references therein, that in $\rn$,
$\H^{n-p}(E) < \infty$ implies that the $p$--capacity of $E \subset \rn$ is zero.

\begin{thm} \label{Tupper}
Suppose that $\mu$ is a doubling measure in $X$ and $E \subset X$. Then 
\begin{equation}  \label{ETu1}
AM(\Gamma(E)) \leq C_{\mu} \, co\H^1(E)
\end{equation}
and for $p > 1$, 
$co\H^p(E)<\infty$ implies 
$M_p(\Gamma(E))=0$. 
\end{thm}

\begin{proof} 
First, we prove
\begin{equation}  \label{ETup} 
AM_p(\Gamma(E)) \leq C_{\mu} \, co\H^p(E)
\end{equation}
for any $1\le p<\infty$.
We may assume that $co\H^p(E) < \infty$. 
For $j = 1, \,2,\, ...$ choose a covering $B(x^j_i,r^j_i), \,
i=1,\,2,\, ...$ , of $E$ such that $r^j_i < 1/j$ and
$$ \sum_i \frac{\mu(B(x^j_i,r^j_i))}{(r_i^j)^p} \leq co\H^p_{1/j}(E)  +\frac{1}{j}.$$
Set
$$\rho_j(x) = \Bigl\{ \sum_i\frac
{1}{(r^j_i)^p}\chi_{B^j_i}(x)\Bigr\}^{1/p}$$
where $B^j_i = B(x^j_i,2r^j_i)$.
Then $\rho_j$ is a Borel function and we show that the sequence $(\rho_j)$ is admissible for 
$\Gamma(E)$. Indeed, if $\gamma \in \Gamma(E)$, then
$\gamma$ meets $E$ and since $\gamma$ is not a constant path, $\diam \, \langle\gamma\rangle > 4/j$ for large $j$
and hence there is $j_0$ such that for $j \geq j_0$ we find $i = i(j)$ such that $\gamma$ meets $B(x^j_i, r^j_i)$
and $X \setminus B^j_i$. Thus $\gamma$ travels in $B^j_i$ at least distance $r^j_i$.
Consequently for $j \geq j_0$
$$ \int_{\gamma} \rho_j \, ds \geq \int_{\gamma } \frac{ \chi_{B^j_{i(j)}} }{r^j_{i(j)}}   \, ds \geq 1$$
and hence
$$\liminf_{j \rightarrow \infty}\int_{\gamma} \rho_j \, ds \geq 1.$$
We obtain
$$ 
\aligned
AM_p(\Gamma(E)) &\leq \liminf_{j \rightarrow \infty}\int_{X} \rho_j^p \, d\mu 
= \liminf_{j \rightarrow \infty}\, \sum_i\frac
{\mu(B^j_i)}{(r^j_i)^p} 
\\&
 \leq C_{\mu} \liminf_{j \rightarrow \infty}\sum_i \frac{\mu(B(x^j_i,r^j_i)) }{(r^j_i)^p} 
\leq C_{\mu} \liminf_{j \rightarrow \infty}\, ( co\H^p_{1/j}(E) + \frac{1}{j}) 
\\&=C_{\mu} \, co\H^p(E),
\endaligned
$$
which proves \eqref{ETup}

Now, for $p=1$ we are done. 
If $p>1$, we know by \cite[Theorem 1]{HEMM}
that $M_p=AM_p$, therefore we have
\eqn{ETupa}
$$
M_p(\Gamma(E)) \leq C_{\mu} \, co\H^p(E).
$$
To prove that $M_p(\Gamma(E))=0$, we first use 
\eqref{ETupa} to
construct a sequence $(\rho_j)$ of $M$--admissible functions for $\Gamma(E)$
such that 
\eqn{CCp}
$$
\int_{X}\rho_j^p\,d\mu\le C\text{ with }C=1+C_{\mu} \, co\H^p(E)
$$ 
and $\mu(\{\rho_j>0\})\to 0$. 
Note that  $\mu(\{\rho_j>0\})$ can be made arbitrary small. To see this let $\varepsilon > 0$ and since $\mu(E) = 0$ we can choose an open
set $G \supset E$ with $\mu(G) < \varepsilon$. 
If $\rho$ is admissible for $\Gamma(E)$, we set 
$$
\tilde\rho=
\begin{cases}
\rho&\text{in }G,\\
0&\text{ in }X\setminus G.
\end{cases}
$$
Each path $\gamma \in \Gamma(E)$
has a subpath $\tilde{\gamma}\in \Gamma(E)$ with locus in $G$. Then
$$
\int_{\gamma}\tilde\rho\,ds \ge \int_{\tilde\gamma}\rho\,ds\ge 1,
$$
and thus $\tilde \rho$ is admissible for $\Gamma(E)$ as well.
Moreover, $\mu(\{ \tilde{\rho}> 0 \}) < \varepsilon$ and
$$ \int_X \tilde{\rho}^p \, d\mu \leq \int_X \rho^p \, d\mu .$$

Now, we select a special subsequence.
We proceed by induction.  Set $m_1=1$. If $m_1,\dots,m_{j-1}$ are determined,
we find $m_j$ such that 
\eqn{2-j}
$$
\int_{E_j}(\rho_{m_1}+\dots+\rho_{m_{j-1}})^p\,d\mu<2^{-j},
$$
holds with $E_j=\{\rho_{m_j}>0\}$.
We claim that 
\eqn{sump}
$$
\int_X(\rho_{m_1}+\dots+\rho_{m_{j}})^p\,d\mu\le 2^{p-1}(Cj+1).
$$
Indeed, it follows from \eqref{CCp} as we prove
\eqn{induc}
$$
\int_X(\rho_{m_1}+\dots+\rho_{m_{j}})^p\,d\mu\le
2^{p-1}\Bigl(\int_X(\rho_{m_1}^p+\dots+\rho_{m_{j}}^p)\,d\mu+\sum_{i=1}^{j}2^{-i}\Bigr)
$$
by induction.
The inequality is trivial for $j=1$.
If it holds for $j-1$, using  \eqref{2-j} we obtain
$$
\aligned
\int_{X}(\rho_{m_1}&+\dots+\rho_{m_{j}})^p\,d\mu
\le \int_{X\setminus E_j}(\rho_{m_1}+\dots+\rho_{m_{j-1}})^p\,d\mu
\\&\qquad+\int_{E_j}(\rho_{m_1}+\dots+\rho_{m_{j}})^p\,d\mu
\\&
\le 2^{p-1}\Bigl(\int_X(\rho_{m_1}^p+\dots+\rho_{m_{j-1}}^p)\,d\mu+\sum_{i=1}^{j-1}2^{-i}\Bigr)
\\&\qquad+2^{p-1}\Bigl(\int_X\rho_{m_j}^p\,d\mu+\int_{E_j}(\rho_{m_1}+\dots+\rho_{m_{j-1}})^p\,d\mu\Bigr)
\\&
\le 2^{p-1}\Bigl(\int_X(\rho_{m_1}^p+\dots+\rho_{m_{j}}^p)\,d\mu+\sum_{i=1}^{j}2^{-i}\Bigr)
\endaligned
$$
which proves \eqref{induc} for $j$.

Finally, we test the $M_p$-modulus of $\Gamma(E)$ by the admissible functions
$$
g_k=\frac{1}k\sum_{j=1}^k\rho_{m_j}.
$$
Then it is evident that each $g_k$ is admissible for $\Gamma(E)$ and 
by \eqref{sump} 
$$
M_p(\Gamma(E)) \le \int_{X}g_k^p\,d\mu\le 2^{p-1}k^{-p}(Ck+1).
$$
\end{proof}

\begin{rmrk}
\label{RAMRn}
Consider the inverse implication in Theorem \ref{Tupper} for $p > 1$ in $\rn$.
Let $E \subset \rn$ be a Borel set with
$M_p(\Gamma(E)) < \infty$, $1 < p \le n$.
If $K \subset E$ is compact, then
$$M_p(\Gamma(K)) \le M_p(\Gamma(E)) < \infty$$
and it easily follows that for all open sets
$\Omega \supset K$
$$ \cp_p(K, \Omega) \le  M_p(\Gamma(K))$$
where $ \cp_p(K, \Omega)$ stands for the
ordinary variational $p$--capacity of the
condenser $(K, \Omega)$, see Section 3 and \cite[Chapter 2]{HKM}.
From \cite[Lemma 2.34]{HKM} it follows that $K$ has $p$--capacity zero and hence by the Choquet capacitability theorem $E$ has also capacity zero. This implies, see e.g. \cite[Theorem 2.27]{HKM}, that
the Hausdorff dimension of $E$ is at most $n-p$ but not that $\H^{n-p}(E) < \infty$.
\end{rmrk}
 
We also need some properties of functions of bounded
variation ($BV$) in $X$, see \cite{Mi} (in metric measure spaces) and 
\cite{AFP} (in the Euclidean spaces). Let $\Omega \subset X$ be open and
denote by $\Lip_{\textnormal{loc}}(\Omega)$ the set of locally Lipschitz
functions in $\Omega$.
Given $u \in L^1_{\textnormal{loc}}(\Omega)$ and an open set $G\subset \Omega$ we 
define
$$ 
V(u,G) = \inf
\Bigl\{ \liminf_i \int_{G} |\nabla u_i| \, d\mu \colon u_i \rightarrow u \textnormal{ in 
} L^1_{\loc}(G)\Bigr\}
$$
Here $|\nabla u(x)|$ stands for the local Lipschitz constant for $u$ at $x$, i.e.
$$ |\nabla u(x)| =  \liminf_{r \rightarrow 0} \sup_{y \in B(x,r)} \frac{|u(y) - u(x)|}{r},$$
see \cite[Section 1.3]{BB}.
A function  has 
\textit{bounded variation} in $\Omega$, $u \in BV(\Omega)$, if 
$V(u,\Omega)<\infty$.

Let $\Omega \subset X$ be open and let $E \subset X$ be measurable. 
The \textit{perimeter} of $E$ in $\Omega$ is 
$P(E,\Omega)=V(\chi_E,\Omega)$ and we write
$P(E)=P(E,X).$

The space $X$ supports the (weak) $BV$--Poincar\'e inequality, see \cite[Remark 3.5]{Mi}, if
\begin{equation} \label{EPBV}
 \int_{B(x,r)}|u -u_{B(x,r)} |\, d\mu \leq C_P \, r\,V(u,B(x,\lambda_P r))
\end{equation}
in each ball $B(x,r)$ and for each $u \in BV(X)$. Here $u_{B(x,r)}$ stands for 
the mean value of $u$ in $B(x,r)$. The constants $C_P\geq 1$ and 
$\lambda_P \geq 1$ are independent of $B(x,r)$ and $u$ and called the Poincar\'e
constants of $X$. Note that (\ref{EPBV}) is a consequence of 
the standard weak Poincar\'e inequality for integrable 
functions with upper gradients, see \cite[Chapter 4]{BB} and \cite{Mi}. 

We use the standard assumptions (A) on the space $X$:
\begin{itemize}
\item $X$ is complete,
\item the measure $\mu$ is doubling,
\item $X$ supports the $BV$--Poincar\'e inequality
(\ref{EPBV}).
\end{itemize}
Note that if $\mu$ is doubling and $X$ is complete,
then $X$ is proper, i.e. closed and bounded subsets of
$X$ are compact, see \cite[Section 3.1]{BB}. Moreover, $X$
is connected \cite[Proposition 4.2]{BB}.

\section{Newtonian and perimeter capacities in $X$}

Throughout this and the next section we assume that $(X,d)$ and $\mu$ satisfy the assumptions (A). 

Let $G$ be a bounded open set in $X$, let $K$ be a compact subset of $G$ and 
let $ \Lip_0(K,G) $ be the set of all Lipschitz functions $u$ with compact support in $G$
satisfying  $u \geq 1$ on $K$. We define 
\eqn{capcomp}
$$ \cp_1(K,G) = \inf \Big\{
\int_G |\nabla u| \, d\mu: \, u \in \Lip_0(K,G) \Big\} .$$
Obviously the infimum does not change if restricted to test functions satisfying $0 \leq u \leq 1$.

It is easy to see that $\Lip_0(K,G) \neq \emptyset$
if $G  \neq \emptyset$ and thus $ \cp_1(K,G) < \infty$. Note that if $G$ is 
compact, then the constant function $1$ is a competitor and thus $ \cp_1(K,G)= 0$.

If $U \subset G$ is open, then we set
$$ \cp_1(U,G) = \sup \{ \cp_1(K,G): \,
K \subset U \textnormal{ compact}\}$$
and for an arbitrary set $E \subset G$
$$  \cp_1(E,G) = \inf \{ \cp_1(U,G): \,
 U \textnormal{ open }, E \subset U \subset G \}.$$ 

Now there are two definitions for $\cp_1(E,G)$ when $E$ is compact 
but since the competitors are continuous the next lemma is immediate. 

\begin{lemma}  \label{Lcoincide}
If $K \subset G$ is compact, then
\begin{equation} \label{Ecoin}
\cp_1(K,G) = \inf \{\cp_1(U,G) : \, U \textnormal{ open, }  K \subset U \subset G\},
\end{equation}
where the capacity on the left is according to \eqref{capcomp}.
\end{lemma}

Next we summarize the main properties of the capacity.
In particular, we show that $\cp_1(\cdot,G)$
defines a Choquet capacity and thus, by the  Choquet capacitability theorem, 
each Suslin (in particular, a Borel) set $E \subset G$ is capacitable.

We also compare  the widely used Newtonian type $p$--capacity 
\begin{equation} \label{ENew}
\widetilde{\cp}_p(E,G) =  \inf_u \int_G (g_u)^p \, d\mu
\end{equation}
for $p = 1$ to $\cp_1(E,G)$. In (\ref{ENew})
the infimum is taken over all (precisely defined)
$u \in N^{1,p}_0(G)$ such that $u \geq 1$ on $E$ and
$g_u$ is the minimal upper gradient of $u$, see
\cite[Section 6.3]{BB}. This is a Choquet capacity if
$p > 1 $ but not in the case $p=1$ because $\widetilde{\cp}_1$
does not satisfy (\ref{i:upwards}) below. For an example see 
\cite[Example 6.18]{BB} where it also becomes evident how
$\cp_1(E,G)$ differs from $\widetilde{\cp}_1(E,G)$.

\begin{prop}\label{p:cap}
\begin{enumerate}[\rm(a)]
\item\label{i:monotone} The set function $E \mapsto \cp_1(E,G)$ is monotone, i.e.
$$
E_1 \subset E_2 \subset G,
  \implies \cp_1(E_1,G) \le \cp_1(E_2,G).  
$$
\item\label{i:compact}
If $K_1,K_2,\dots\subset G$ are compact
and $K_1\supset K_2\supset\dots$,
then
$$
\cp_1\Big( \bigcap_{j=1}^{\infty}K_j , G\Big)
=\lim_{j\to\infty} \cp_1(K_j, G).
$$
\item\label{i:tilde}
$
\cp_1(E,G)\le \widetilde{\cp}_1(E,G)
\text{ and }\cp_1(K,G)= \widetilde{\cp}_1(K,G)
\text{ if }K\text{ is compact} .
$
\item\label{i:strong} 
If $K_1,K_2$ are compact, then 
$$
\cp_1(K_1\cup K_1,G)+\cp_1(K_1\cap K_2,G)\le \cp_1(K_1,G)+\cp_1(K_2,G).
$$
\item \label{i:upwards}
$
 E_1 \subset E _2 \subset \dots \subset G 
\implies \cp_1 \Big(\bigcup_{j=1}^{\infty}E_j,G \Big)
=\lim_{j\to\infty}\cp_1(E_j,G).
$
\item\label{i:Choquet} 
If $E\subset G$ is Suslin, then
$$
\cp_1(E,G) = 
\sup 
\{ \, \cp_1(K,G) : \, K \subset E \textnormal{ compact } \}.
$$
\end{enumerate}
\end{prop}

\begin{proof}
The properties \eqref{i:monotone} and \eqref{i:compact} are obvious.
The inequality in \eqref{i:tilde} is obvious if $E$ is open; for the
case of $E$ arbitrary we use \cite[Theorem 6.19 (vii)]{BB}
(note that the symbol $\cp_1$ stands for $\widetilde{\cp_1}$ in
\cite{BB}).
The equality for $K$ compact follows from \cite[Theorem 6.19 (x)]{BB}.
The property \eqref{i:strong} follows from \cite[Theorem 6.17 (iii)]{BB}
taking into account the equality in \eqref{i:tilde}.
Now, the properties \eqref{i:upwards} and \eqref{i:Choquet} are obtained
using the general theory of capacities developed by Choquet in \cite{Cho},
see also \cite{Br}, \cite{Kech}.
\end{proof}

If $G$ is a bounded open set in $X$ and $K \subset G$ compact, then we denote by $\Gamma(K,G)$ the family of all paths in $X$ which connect $X \setminus G$ to $K$.

\begin{lemma}\label{Lcap}
If $G$ is a bounded open set in $X$ and $K \subset G$ compact, then
$$
 \cp_1(K,G) = M_1(\Gamma(K,G)) =  AM(\Gamma(K,G)).
$$
\end{lemma}

\begin{proof}
Since for each function $u \in \textnormal{Lip}_0(K,G)$, $|\nabla u|$ is $M$--admissible
for the family $\Gamma(K,G)$, we
have $M_1(\Gamma(K,G)) \leq \cp_1(K,G)$. For the converse inequality we use the method in
\cite[Section 5.2]{BB}. Let $\rho$ be $M$--admissible for $\Gamma(K,G)$ and $\varepsilon > 0$. We may assume
that $\rho$ is lower semi--continuous.
From Lemmata 5.25 and 5.26 in \cite{BB} it follows that the function $\rho + \varepsilon$
is an upper gradient of the lower semi--continuous function 
$$ u(x) = \min \Big(1, \inf_{\gamma}  \int_{\gamma} (\rho +\varepsilon) \, ds \Big)$$
 in $G$. Here the infimum is taken over all paths connecting $X 
 \setminus G$ to $x \in G$. Moreover, $u = 0$ in $X\setminus G$ and $u = 1$ in $K$. 
Using  Proposition \ref{p:cap}\eqref{i:tilde} we obtain 
$$\cp_1(K,G) =\widetilde{\cp_1}(K,G)\leq \int_G (\rho +\varepsilon) \, d\mu \leq \int_G \rho \, d\mu 
+ \varepsilon\, \mu(G)$$
and letting $\varepsilon \rightarrow 0$ we obtain the desired inequality.

For the second equality it suffices to show that $M_1(\Gamma(K,G)) \leq  AM(\Gamma(K,G))$ because $M_1(\Gamma) \geq AM(\Gamma)$ for every path family $\Gamma$ in $X$. Let 
$\Gamma(K,G,L)$ denote the family of all the paths $\gamma$
in $\Gamma(K,G)$ whose length $\ell$ satisfies $ \ell \le L$. Note that 
\begin{equation}
\label{EM1}
M_1(\Gamma(K,G))=\sup_{L}M_1(\Gamma(K,G,L)).
\end{equation}
Indeed, if $\rho$ is admissible for $\Gamma(K,G,L)$, then
$\rho + \frac{1}{L}\chi_G$ is admissible for $\Gamma(K,G)$.

Fix $L$. Each $\gamma\in\Gamma(K,G,L)$ has a reparametrization $\xi\colon [0,L]\to X$
which is a curve with  $\Lip\xi\le1$; we denote the set of all such reparametrizations by $\Xi(K,G,L)$.
For a Borel set $E\subset X$ set
$$
\nu_{\xi}(E)=\int_{\xi}\chi_E\,ds.
$$
Set $\E=\{\nu_{\xi}\colon \xi\in \Xi(K,G,L)\}$. Let $\K$ be the weak* closure of $\E$.
Then 
\eqn{amc}
$$
AM_c(\Gamma(K,G,L))=AM_c(\Xi(K,G,L))=AM_c(\E)=AM_c(\K).
$$
Only the last equality is not obvious. Let $(\xi_j)$ be a sequence of curves from 
$\Xi(K,G,L)$ such that $\nu_{\xi_j}$ converge weak* to $\nu\in\K$. By the Arzel\`a-Ascoli theorem 
(see \cite[p. 169]{Roy})
there exists a subsequence (not relabelled) which converges uniformly to a limit curve
$\xi$, and, by compactness of 
$K$ and openness of $G$, we have $\xi\in \Xi(K,G,L)$. For each non-negative continuous function $\rho$ on
$X$ we have
$$
\int_{\xi}\rho\,ds\le \liminf_j \int_{\xi_j}\rho\,ds= \lim_j\int_{X}\rho\,d\nu_{\xi_j}=\int_{X}\rho\,d\nu.
$$
It follows that each admissible sequence for $AM_c(\E)$ is also admissible for $AM_c(\K)$ and 
thus $AM_c(\K)\le AM_c(\E)$, whereas the converse inequality is obvious. This proves \eqref{amc}.
By \cite[Theorem 5.5]{HEKMM}, $AM(\K)=M_1(\K)$ (as $\K$ is compact) and 
by \cite[Theorem 3.4]{HEKMM}, $AM=AM_c$.
Hence
$$
\aligned
M_1(\Gamma(K,G,L))&\le M_1(\K)=AM_c(\K)=AM_c(\Gamma(K,G,L))\\&
=AM(\Gamma(K,G,L))\le AM(\Gamma(K,G)).
\endaligned
$$
Passing to the supremum over $L$ we obtain the conclusion.
\end{proof}

\begin{lemma}\label{LSus} If $E \subset G$ is a Suslin set, then 
$\cp_1(E,G) \leq AM(\Gamma(E)).$
\end{lemma}

\begin{proof}
Since $E$ is a Suslin set, Proposition \ref{p:cap}(f) implies that there are compact sets
$K_1 \subset K_2 \subset \, ... \, \subset E$ such that
$\cp_1(E,G) = \lim_i \cp_1(K_i,G)$. Now by Lemma \ref{Lcap} 
$$\cp_1(K_i,G) = AM(\Gamma(K_i,G)) \leq AM(\Gamma(E))$$
because $\Gamma(K_i,G) \subset \Gamma(E)$.
\end{proof}

\begin{lemma} \label{Linc}
Let $K_1 \subset K_2 \subset \, ... \, $ be compact sets in G with
\begin{equation} \label{Ecapfin}
\lim_{i \rightarrow \infty}\cp_1(K_i,G) < \infty.
\end{equation}
Then there is a $BV$ function $w$ in $X$ such that $w = 0$ in $X \setminus G$,
$w =1$ on $\bigcup_i K_i$, $0 \leq w \leq 1$ and 
\begin{equation} \label{Ecapfin1}
V(w,X) \leq \lim_{i \rightarrow \infty} \cp_1(K_i,G).
\end{equation}
\end{lemma}

\begin{proof}
For each $i$ pick $u_i \in \Lip_0(K_i,G)$ such that $0\le u_i\le 1$ 
and 
$$ 
\int_G |\nabla u_i|\, d\mu \leq \cp_1(K_i,G) + 1/i.
$$
By the compact embedding of $BV$ into $L^1_{\loc}$, see  \cite[Theorem 
3.7]{Mi}, there is a limit function $w$ and a subsequence $(v_i)_i$ of $(u_i)_i$ such that
$v_i\to w$ in $L_{\loc}^1(X)$ and $\mu$-a.e. In particular, we can assume that $w=1$ 
on $\bigcup_i K_i$ and \eqref{Ecapfin1} holds.
\end{proof}

We recall some measure theoretic notation. Let $E \subset X$ be a ($\mu$--) measurable set. The \textit{measure
theoretic boundary} $\partial_*E$ of $E$ consists of points
$x \in X$ such that
$ \Theta(x,E) > 0$
and $\Theta(x,X \setminus E) > 0$
where
$$\Theta(x,A) = \limsup_{r \rightarrow 0}\frac{\mu(B(x,r) \cap A)}{\mu(B(x,r))} $$
is the upper $\mu$--density of $A$ at $x$.
 The \textit{measure
theoretic interior} $\inte E$ 
and the \textit{measure theoretic exterior} $\exte E$ of $E$
are the sets
of points $x \in  X$ where $\Theta(x,X \setminus E) =0$ and
$\Theta(x,E)=0$, respectively. 
The sets $\partial_*E$, $\inte E$ and $\exte E$ are Borel sets.

For an open bounded set $G \neq X$ and $E \subset G$ we define the \textit{perimeter capacity} of $E$ in 
$G$ as 
$$\Cp(E,G) = \inf\Big\{P(F,X) \colon E \subset \inte F, \,F \subset G \textnormal{ measurable} \Big\}.$$
Note that the perimeter of $F$ is relative to $X$ and not relative to $G$.

\begin{lemma}\label{cap=Cap}
If $E$ is a Suslin set in $G \subset X$ and $AM(\Gamma(E)) < \infty$, then 
\begin{equation} \label{ECc1}
 \Cp(E,G) \le \,\cp_1(E,G).
\end{equation}
\end{lemma}

\begin{proof}
Let $U$ be an open set such that $E \subset U \subset G$.
By Lemma \ref{LSus} we have $\cp_1(E,G) < \infty$. 
Next choose compact sets $K_1 \subset K_2 \subset \, ... \subset U$ such that 
$\bigcup_i K_i = U$; now 
$$ \cp_1(K_i,G) \leq \cp_1(U, G)$$ 
for all $i$. 

Let $w$ be the $BV$ function in Lemma \ref{Linc}. Note that $w = 1$ in $U = \bigcup_i K_i$.
By the co--area formula \cite[Proposition 4.2]{Mi}
and Lemma \ref{Linc}
$$\int_0^1P(\{x: w(x) > t \}, X) \,dt \le V(w,X) \leq \lim_i \cp_1(K_i, G) \leq \cp_1(U, G).$$
 Thus there
is some $t \in (0,1)$ such that the set $A =\{x: w(x) > t \}$
has finite perimeter, $\inte A \supset E$ and
$P(A, X) \leq \cp_1(U, G)$. Note that it is possible that $A = G$. Since 
$$\Cp(E,G) \leq P(A,X) \leq  \cp_1(U, G)$$
and this holds for all open sets $U$ with $E \subset U \subset G$ we obtain (\ref{ECc1}).
\end{proof}

\section{$AM(\Gamma(E)) \leq C \, co\H^1(E)$}

Throughout this section we assume that $(X,d)$ and $\mu$ satisfy the assumptions (A).

We need the following auxiliary lemma for the main result. Note that the set $E$ below is an arbitrary subset of $X$.

\begin{lemma}\label{L=0} If $E \subset X$ and $AM(\Gamma(E)) < \infty$, then $\mu(E) = 0$.
\end{lemma}

\begin{proof}
By \cite[Theorem 2]{HEMM2} there is a co--Suslin set $E' \supset E$ such that $AM(\Gamma(E')) = AM(\Gamma(E)$. 
Since co--Suslin sets are $\mu$--measurable we may assume that $E$ is $\mu$ measurable and since we can also assume that $E$ is bounded, it suffices to prove the lemma in the case $\mu(E) < \infty$.

Let $\varepsilon > 0$. Since $\mu(\overline{B}(x,r) \setminus B(x,r))=0$ except for a countable set of $r>0$ we find by the Vitali covering theorem disjoint
balls $\overline{B}(x_i, r_i)$ such that
$r_ i < \varepsilon$ and $\bigcup_i B(x_i, r_i)\supset E \setminus E_0$ where $\mu(E_0) = 0$. Now we can replace $E$ by $E \setminus E_0$ which we continue to denote by $E$.

Fix $B_i = B(x_i, r_i)$ and let $K \subset E \cap B_i$ be compact. For $\delta > 0$ pick $u \in N^{1,1}_0(B_i)$ such that $u = 1$
on $K$, $0 \leq u \leq 1$ and
$$\int_{B_i}g_u \, d\mu < \cp_1(K,B_i) + \delta.$$
By the Poincar\'e inequality \cite[Theorem 5.51]{BB} for $N^{1,1}_0(B_i)$--functions there is a 
constant $C$ depending only on $C_P$ and $C_{\mu}$ so that
$$\mu(K) \leq \int_{B_i}u \, d\mu \leq C r_i 
\int_{B_i} g_u \, d\mu <
C r_i  (\cp_1(K,B_i) + \delta)$$
and letting $\delta \rightarrow 0$ we obtain 
from Lemma \ref{Lcap}
$$\mu(K) \leq C r_i  AM(\Gamma(K,B_i)) \leq
 Cr_i AM(\Gamma(E\cap B_i,B_i)). $$
Since this holds for all compact sets $K \subset E\cap B_i $
$$\mu(E\cap B_i) \leq C r_i \, AM(\Gamma(E\cap B_i,B_i)).$$
The path families $\Gamma(E\cap B_i,B_i)$
lie in the disjoint sets $\overline{B}_i$ and
are subfamilies of $\Gamma(E)$. Summing over $i$ we obtain
$$ \mu(E) = \sum_i \mu(E\cap B_i) \leq C \sum_i
r_i \, AM (\Gamma(E\cap B_i,B_i)) \leq 
C \varepsilon \, AM(\Gamma(E)),$$
and $\varepsilon \rightarrow 0$ completes the proof.
\end{proof}

The comparison of the $BV$ capacity with the $(n-1)$--dimensional 
Hausdorff content is due to Fleming \cite{Fle}. It has been 
generalized to the framework of metric measure spaces by  Kinnunen, 
Korte, Shanmugalingam and Tuominen \cite{KKST1}. Here we need a 
version for the $\delta$-Hausdorff content related to the $co\H^1$--measure.

\begin{lemma}
\label{Lnull}
Let $M$ be a bounded open set in $X$.
For $\delta>0$ there exists
$\alpha>0$ such that for each 
open set
$G$ with $\mu(G)<\alpha$ and $E\subset G\subset M$
we have
\begin{equation} \label{ELnull}
co\H^1_{\delta}(E) \le C\, \Cap(E,G),
\end{equation}
where $C$ depends only on $C_P$, $\lambda_P$ and $C_{\mu}$.
\end{lemma}

\begin{proof} 
We write for $C$ a generic constant which depends only on $C_P$, $\lambda_P$ and $C_{\mu}$.

Set $\delta' = \delta/(5 \lambda_P)$ and $\kappa = 4 C_P$.  
Let $G$ be a bounded open set such that $E \subset G \subset M$. We find
$\alpha > 0$ such that for each $x \in G$
\begin{equation} \label{EG1}
\mu(B(x,\delta') \cap G) \leq \frac{1}{\kappa}
\mu(B(x, \delta'))
\end{equation}
provided that
$\mu(G) < \alpha$. Suppose that no such  $\alpha$ exists. Then there are
open sets $G_i$ and $x_i \in G_i$ such that $E \subset G_i \subset M$ and
$$
\frac{1}{i}> \mu(G_i) \geq \mu(B(x_i,\delta') \cap G_i) > \frac{1}{\kappa}\mu(B(x_i, \delta'))
$$
but because each $x_i$ belongs to a fixed bounded set $M$, $\mu(B(x_i, \delta')) > c > 0$ which
 leads to contradiction.
 
Fix $G$ as above. To prove (\ref{ELnull}) we may assume that $\Cap(E,G) < \infty$ and
for $\varepsilon > 0$ we choose a competitor $F \subset G$ for $\Cp(E,G)$ with
$P(F,X) \leq \Cp(E,G) + \varepsilon$. Let $x \in E$, $B(r) = B(x,r)$ and define
$$r_x = \inf \{r > 0: \, \mu(F \cap B(r)) \leq
 \frac{1}{2C_P} \mu(B(r) \}.$$
 Now $0 < r_x < \delta'$ because
$$ \lim_{r \rightarrow 0} \frac{\mu(F \cap B(r))}{\mu(B(r))} = 1$$
and by (\ref{EG1})
$$\mu(F \cap B(\delta')) \leq \mu(G \cap B(\delta') )
\leq \frac{1}{4C_P}\mu(B(\delta')) <
\frac{1}{2C_P}\mu(B(\delta')).$$

Let $r < r_x$. Then
$$\mu(F \cap B(r_x)) \geq \mu(F  \cap B(r))
> \frac{1}{2C_P} \mu(B(r))$$
and letting $r \rightarrow r_x$ we obtain
\begin{equation} \label{EG2}
\mu(F \cap B(r_x)) \geq \frac{1}{2C_P} \mu(B(r_x)).
\end{equation}
On the other hand we show that
\begin{equation} \label{EG3}
\mu(F \cap B(r_x)) \leq \frac{1}{2} \mu(B(r_x)).
\end{equation}
If $\mu(F \cap B(r_x)) \leq  \mu(B(r_x))/(2C_P)$, then equality holds 
in (\ref{EG2}) and (\ref{EG3}) is immediate. If
$$\mu(F \cap B(r_x)) > \frac{1}{2C_P} \mu(B(r_x))$$
then by the definition of $r_x$ there is
$r \in (r_x, 2r_x)$ such that
$$\mu(F \cap B(r_x)) \leq \mu(F \cap B(r)) \leq \frac{1}{2C_P} \mu(B(r)) \leq \frac{1}{2} \mu(B(r_x)).$$

Next we use the $BV$--Poincar\'e inequality (\ref{EPBV}) for the $BV$ function $\chi_F$ in $B(r_x)$. 
By (\ref{EG2}) and (\ref{EG3})
$$\frac{1}{2C_{\mu}} \leq (\chi_F)_{B(r_x)} = \frac{\mu(F \cap B(r_x))}{\mu(B(r_x))} \leq \frac{1}{2} $$
and we obtain
$$
\aligned
\frac{\mu(B(r_x))}{4 C_{\mu}}& \leq 
\frac{ \mu(F \cap B(r_x))}{2} \leq \int_{F \cap B(r_x)} (1-  (\chi_F)_{B(r_x)} ) \, d\mu
\\
&\leq \int_{B(r_x)} |(\chi_F-  (\chi_F)_{B(r_x)} | \, d\mu \leq C_P \, r_x P(  F, B(\lambda_P \, r_x))
\endaligned$$
and so
\begin{equation} \label{EG4}
\frac{\mu(B(r_x))}{r_x} \leq C \, P(  F, B(\lambda_P r_x)).
\end{equation}
By the $5$--covering lemma we find balls $B_j = B(x_j, \lambda_P r_{x_j})$ from 
the collection $\{ B(x,\lambda_P r_x) \}$ so that 
the balls $B_j$ are disjoint and the balls $5B_j = B(x_j, 5\lambda_P r_{x_j})$ cover
$E$. Set $D = \bigcup_j 5 \,B_j$. Since $ 5 \lambda_P r_{x_j} < 5\, \lambda_P \delta' = \delta$ 
we obtain from (\ref{EG4})
$$ co\H^1_{\delta}(E)  \leq
\sum_j \frac{\mu(5B_j)}{5r_{x_j}} \leq
C \sum_j \frac{\mu(B(x_j, r_{x_j}))}{r_{x_j}}$$
$$\leq C \, \sum_j P(F, B_j)
\leq  C \, P(F, X) \leq C (\Cap(E,G) + \varepsilon)$$
where the doubling property of $\mu$  and the fact that the balls $B_j$ are disjoint have also been used.
Letting $\varepsilon \rightarrow 0$ we complete the proof.
\end{proof}

The following lemma combines the achieved results.

\begin{lemma} \label{Lcomp}
Suppose that $E \subset X$ is a bounded Suslin set such that $AM(\Gamma(E)) < \infty$. Then
\begin{equation} \label{Ecompp}
co\H^1(E) \le C \, AM(\Gamma(E)) 
\end{equation}
where the constant $C$ depends only on $C_P, \, \lambda_P$ and $C_{\mu}$.
\end{lemma}

\begin{proof}
Lemma \ref{L=0} yields $\mu(E) = 0$. Fix $\delta > 0$ and then, by Lemma \ref{Lnull}, we 
find a bounded open set $G \neq X$ containing $E$ with
$$co\H^1_{\delta}(E)\le C \, \Cp(E,G).$$
Now Lemmata \ref{cap=Cap} and
\ref{LSus} imply
$$
\Cp(E,G) \leq \cp_1(E,G) \leq AM(\Gamma(E))$$
and hence $co\H^1_{\delta}(E) \le C \, AM(\Gamma(E)).$ Passing to 
the supremum w.r.t. $\delta>0$ we obtain (\ref{Ecompp}).
\end{proof}

\begin{thm} \label{Tsigma}
Let  $E \subset X$ be a Suslin set. Then
\begin{equation} \label{Ecomp}
C_1 \, co\H^1(E) \le  \, AM(\Gamma(E)) \leq C_2 \, co\H^1(E)
\end{equation}
where the constant $C_1 > 0$ depends only on $C_P, \, \lambda_P$ and $C_{\mu}$ 
and the constant $C_2$ only on $C_{\mu}$.
\end{thm}

\begin{proof}
The second inequality in (\ref{Ecomp}) follows from Theorem \ref{Tupper}. For the first inequality 
fix $x_0\in X$ and observe that
$$
C_1 co\H^1(E\cap B(x_0,j))\le AM(\Gamma(E\cap B(x_0,j))) \le  
AM(\Gamma(E)),\qquad j=1,2,\dots
$$
by Lemma \ref{Lcomp}.
Letting $j\to \infty$ we conclude the proof.
\end{proof}

If $E \subset X$ has $\sigma$--finite $co\H^1$--measure,
then Theorem \ref{Tsigma} holds without the assumption that 
$E$ is a Suslin set.

\begin{thm} \label{Tsigma1}
Suppose that $E \subset X$ has $\sigma$--finite $co\H^1$--measure. Then 
\begin{equation} \label{Ecomp1}
C_1 \, co\H^1(E) \le  \, AM(\Gamma(E)) \leq C_2 \, co\H^1(E)
\end{equation}
where the constants $C_1 $ and $C_2$ are as in Theorem \ref{Tsigma}.
\end{thm}

\begin{proof}
The right inequality of (\ref{Ecomp1}) again follows from Theorem
\ref{Tupper}. 
For the left inequality suppose first that $co\H^1(E)<\infty$. Then 
there is a Borel set $F\supset E$ such that $co\H^1(F)=co\H^1(E)$ and a 
co-Suslin set $E'\supset E$ such that $AM(\Gamma(E'))=AM(\Gamma(E))$,
see \cite[Theorem 2]{HEMM2}. We may assume that $E'\subset F$.
Then the set function 
$$
\nu\colon A\mapsto co\H^1(A\cap F),\qquad A\text{ Borel}
$$
is a finite Borel measure. We extend $\nu$ to 
the class of all $\nu$-measurable sets by completion. Then the set
$E'$ is $\nu$-measurable as it is co-Suslin \cite[Theorem 21.10]{Kech}. 
It follows that there
is a Borel set $A\subset E'$ such that $\nu(A)=\nu(E')$
\cite[Theorem 17.10]{Kech}.
Now, 
$$
co\H^1(E)\le co\H^1(E')=\nu(E')=\nu(A)=co\H^1(A)
$$
and 
$$
AM(\Gamma(A))\le AM(\Gamma(E'))=AM(\Gamma(E)).
$$
Since $C_1\,co\H^1(A)\le AM(\Gamma(A))$, we conclude that
$$
C_1\,co\H^1(E)\le AM(\Gamma(E)).
$$

In the general case we find $E_1\subset E_2\subset\dots$ such that 
$co\H^1(E_i)<\infty$ and $E=\bigcup_i E_i$. Let 
$F_i$ be Borel set such that $F_i\supset E_i$ and 
$co\H^1(F_i)=co\H^1(E_i)$.
Since $E_1\subset F_1\cap F_2\subset F_1$, we have 
$co\H^1 (F_1\setminus F_2)=co\H^1(F_1)-co\H^1 (F_1\cap F_2)=0$
and thus $co\H^1(F_1\cup F_2)\le co\H^1(E_2)$. Continuing by induction we 
may assume that $F_1\subset F_2 \subset \dots$. Therefore
$$
co\H^1(E)\le co\H^1\Big(\bigcup_iF_i\Big)= \lim_i co\H^1(F_i)= \lim_i 
co\H^1(E_i)\le C_1^{-1}AM(\Gamma(E)).
$$
\end{proof}

In the Euclidean setting, the $co\H^1$ measure satisfies
$$\alpha_{n-1}co\H^{1}(E)=\alpha_n\haus(E),$$ 
where
$$\haus(E) = \sup_{\delta> 0}\haus_{\delta}(E)$$
is the spherical Hausdorff measure defined through the 
spherical Hausdorff $\delta$-content
$$
 \haus_{\delta}(E)= \inf \Bigl\{
 \sum_{i=1}^{\infty} \alpha_{n-1}r_i^{n-1} \colon\,
 E \subset \bigcup_{i=1}^{\infty} B(x_i,r_i)\,, \, r_i<\delta \Bigr\}
 $$
and
$\alpha_m$ denotes the volume of the $m$-dimensional unit ball.
It is easily seen that the spherical Hausdorff measure is equivalent to the standard Hausdorff
measure $\standhaus$ defined in terms of diameters, namely
$$
\standhaus(E)\le \haus(E)\le 2^n\standhaus(E),\qquad E\subset\rn,
$$
see \cite[2.10.2]{Fed}.
Now, Theorems \ref{Tsigma} and  \ref{Tsigma1} 
yield (with properly modified constants):

\begin{cor} \label{Crn}
If $E$ is a Suslin set in $\rn$ or has $\sigma$--finite $\standhaus$--measure, then
$$C_1 \, \standhaus(E) \leq AM(\Gamma(E)) \leq C_2 \, \standhaus(E)$$
where the positive constants $C_1$ and $C_2$ depend only
on $n$.
\end{cor}

\section{perimeter
 and  $AM$--modulus in $X$}

We  characterize sets $E$ of finite perimeter in $X$ using the $AM$--modulus
of the path family  $\Gamma(\partial_*E)$.
Such a characterization was presented for $X= \rn$  in \cite{HEMM2}. 

We also study the connection of the perimeter of $E$ in an open set $\Omega \subset X$ to the family $\Gamma_{\mathrm{cross}}(E, \Omega)$ whose paths lie in an open set $\Omega$ and  meet both the measure theoretic exterior and interior
of $E$ and present a measure theoretic version of the elementary topological fact. Namely, if $X$ is a topological space, $E \subset X$ and int$\, E$, ext$\,E$ and $\partial E$ are the (topological) interior, exterior and boundary of $E$, respectively, then every curve $\gamma:[a,b] \rightarrow X$ which meets int$\, E$ and ext$ \,E$ also meets $\partial E$. We show that $AM$ a.e. path 
$\gamma \in \Gamma_{\mathrm{cross}}(E, \Omega)$ meets the measure theoretic boundary $\partial_*E$ of $E$ provided that $E$ has finite perimeter in $\Omega$. In \cite[Theorem 5.3]{KL} a closely related result is
proved under more restrictive assumptions on $E$ for the $M_1$--modulus. 

We assume that $X$ satisfies (A) and, as before, $C$ is a constant 
which depends only on $C_{\mu}$, $C_{\lambda}$ and $C_P$ and can change inside a line. 

\begin{lemma} \label{LPAM}
If $\Omega$ be an open set in $X$ and $E \subset X$ measurable, then
$$AM(\Gamma_{\mathrm{cross}}(E, \Omega)) \leq C\,P(E,\Omega).$$
\end{lemma}

\begin{proof}
Let $u$ be the Lebesgue representative of $\chi_E$,
i.e. 
$$u(x) = \lim_{r \rightarrow 0}\frac{\mu(E \cap B(x,r))}{\mu(B(x,r))}$$
whenever the limit exists, then $u(x) = 1, \, x \in \inte E$, $u(x) = 0, \, x \in \exte E$ and $u = \chi_E$ a.e. in $\Omega$.

For the proof we may assume that $P(E, \Omega) < \infty$ and then we can use the special sequence of locally
Lipschitz functions constructed in \cite[Proposition 4.1]{KKST}; i.e. there is a sequence $u_k \in \textnormal{Lip}_{\textnormal{loc}}(\Omega)$ such that
$u_k \rightarrow u$ pointwise $co\H^1$ a.e. in $\Omega \setminus \partial_*E$, $u_k \rightarrow u$
in $L^1(\Omega)$ and
\begin{equation} \label{EPAM1}
\liminf_{k \rightarrow \infty} \int_{\Omega}|\nabla u_k| \, d\mu \leq C \, P(E, \Omega).
\end{equation}

Let $A \subset \Omega \setminus \partial_*E$ be the set where $\lim_k u_k(x) \neq u(x)$. Now
$co\H^1(A) = 0$ and by Theorem \ref{Tupper},
$AM(\Gamma(A)) = 0$.  The sequence of functions $|\nabla u_k|$ is $AM$--admissible for
$\Gamma_{\mathrm{cross}}(E, \Omega) \setminus \Gamma(A)$
since if $\gamma \in
\Gamma_{\mathrm{cross}}(E, \Omega) \setminus \Gamma(A)$ then there are points
$t_1, \, t_2 \in [0, \ell]$ such that $\gamma(t_1)
\in \inte E$, $\gamma(t_2) \in \exte E$ and
$$1 = \lim_{k \rightarrow \infty}|u_k(\gamma(t_1)) -
u_k(\gamma(t_2))| \leq \liminf_{k \rightarrow \infty}\int_{\gamma} |\nabla u_k| \,ds.$$
By (\ref{EPAM1})
$$AM(\Gamma_{\mathrm{cross}}(E, \Omega)\setminus \Gamma(A)) \leq \liminf_{k \rightarrow \infty} \int_{\Omega} |\nabla u_k| \,d\mu \leq C\,P(E,\Omega)$$
and since $AM(\Gamma(A)) = 0$ we have
$$AM(\Gamma_{\mathrm{cross}}(E, \Omega)) \leq C \, P(E,\Omega).$$
\end{proof}

\begin{thm} \label{Tomit}
If $P(E,\Omega) < \infty$ then $AM$ a.e. path $\gamma \in \Gamma_{\mathrm{cross}}(E,\Omega)$ meets $\partial_*E$.
\end{thm}

\begin{proof}
Let $\Gamma$ be the family of paths in $\Gamma_{\mathrm{cross}}(E, \Omega)$ which do not meet
$\partial_*E$.
By the subadditivity of the $AM$--modulus we may assume
that $\Omega$ is bounded. By \cite[Theorem 4.4 and Theorem 4.6]{AMP} for every open
set $G \subset \Omega$
$$ P(E,G) = \int_{\partial_*E \cap G}\theta \, dco\H^1$$
where $\theta = \theta_E$ is a Borel function
with $1/C \leq \theta \leq C$ in $\Omega$
and, moreover, $co\H^1(\partial_*E \cap \Omega)< \infty$.
Let $\varepsilon > 0$. Now we find a compact set $K \subset \partial_*E \cap \Omega$ such that $P(E,G) < \varepsilon$ for $G = \Omega \setminus K$.

Next observe that $\Gamma \subset \Gamma_{\mathrm{cross}}(E, G)$ because each $\gamma \in \Gamma$ does not meet $K$. By Lemma \ref{LPAM}
$$AM(\Gamma ) \leq AM( \Gamma_{\mathrm{cross}}(E, G))  \leq C \, P(E,G) \leq C \, \varepsilon$$
and letting $\varepsilon \rightarrow 0$ we complete the
proof.
\end{proof}

\begin{thm} \label{TAhlPAM}
Suppose that  $E \subset X$ is a ($\mu$--) measurable set. Then for each open set $\Omega \subset X$ 
\begin{equation} \label{EAPAM}
C_1 P(E, \Omega) \leq AM(\Gamma(\partial_*E \cap \Omega)) \leq C_2 P(E, \Omega)
\end{equation}
where the constants $C_1$ and $C_2$ depend only
on $C_P$, $C_{\lambda}$ and $C_{\mu}$.
\end{thm}

\begin{proof}
For the right inequality in (\ref{EAPAM}) we may assume that 
$P(E, \Omega) < \infty$ and
then by \cite[Theorem 4.4]{AMP}, $$co\H^1(\partial_*E \cap \Omega) \leq C \, P(E, \Omega)$$
and now Theorem \ref{Tupper} gives the required inequality. 

For the left side of (\ref{EAPAM}) we note that  $\partial_*E \cap \Omega$ is a Borel set and thus Theorem  \ref{Tsigma} yields
$$ co\H^1(\partial_*E \cap \Omega)) \leq
C \, AM(\Gamma(\partial_*E \cap \Omega)) < \infty .$$
By the recent result of Lahti \cite[Theorem 1.1]{La} this implies that
$P(E, \Omega) < \infty$ and we can  apply again  \cite[Theorem 4.4]{AMP} to conclude 
$$P(E, \Omega) \leq C \, co\H^1(\partial_*E \cap \Omega)) $$
and complete the proof.
\end{proof}

\section{Geometry of level sets in $X$}

The results in the previous sections can be used to study the structure of level sets of $BV$ 
and continuous functions in $X$ and the latter case together with the results in Section 4
produces a plenitude of open sets in $X$ with $co\H^1$ finite boundaries.

We assume that $X$ satisfies the hypotheses (A) and recall some measure theoretic concepts
asociated with $BV$--functions.

For a measurable set $E$ and $x \in X$ we let
$$\overline{D}(E,x) = \limsup_{r \rightarrow 0}\frac{\mu(E \cap B(x,r))}{\mu(B(x,r))},\,
\underline{D}(E,x) = \liminf_{r \rightarrow 0}\frac{\mu(E \cap B(x,r))}{\mu(B(x,r))},$$
and $D(E,x) = \overline{D}(E,x)$ if $\overline{D}(E,x) = \underline{D}(E,x)$.

Let $\Omega$ be an open set in $X$ and $u\in BV(\Omega)$. The upper 
and lower \textit{approximate limits} of $u$ at $x \in \Omega$
are
$$ u^+(x) = \inf\{s: \, D(\{u >s \},x) = 0\} \textnormal{ and }
u^-(x) = \sup\{t: \, D(\{u <t \},x) = 0\}.$$
Then it is immediate that $u^-(x)\le u^+(x)$.
The function $u$ is \textit{approximately continuous} at $x$ if
$u^+(x) = u^-(x)=u(x)$. This holds
a.e.\ in $\Omega$ by the Lebesgue differentiation theorem. 
The set $J_u = \{u^- < u^+ \}$ is called the jump set of $u$
and it has zero $\mu$--measure, see \cite{KKST}.

For $-\infty\le s,t,\le \infty$ we consider the measure theoretic level sets of
$u \in BV(\Omega)$
$$
\aligned
E^t&=\{x\in\Omega\colon u^-(x)\le t\},\\
E_s&=\{x\in\Omega\colon u^+(x)\ge s\},\\
E_s^t&=E_s\cap E^t,\\
\Lambda_t &= E_t^t.
\endaligned
$$

\begin{lemma} \label{Lorder}
If $u \in BV(\Omega)$, then
\begin{equation} \label{Eorder1}
\mu( \Lambda_t) = 0, 
\end{equation}
and consequently $P(E_t,\Omega)=P(E^t,\Omega)$,
for a.e. $t \in \er$.

 If $u$ is (approximately) continuous at $x$, then
$x\in \Lambda_{u(x)}$.
\end{lemma}

\begin{proof}
To prove (\ref{Eorder1}) note that 
$\Lambda_t \subset A_t \cup J_u$,
where
$$A_t = \{ x\in\Omega\colon \, t = u^{-}(x) = u^{+}(x)\}.$$
Since $A_t \cap A_{t'} = \emptyset$ for $t \neq t'$  and $\mu(J_u) =0$, (\ref{Eorder1})
follows. If $\mu( \Lambda_t) = 0$, then $E_t$ differs from $\Omega\setminus E^t$
by a $\mu$-null set and thus $P(E_t,\Omega)=P(E^t,\Omega)$.

If $u$ is approximately continuous at $x$ and $t=u(x)$, then
$t=u^+(x)=u^-(x)$ and thus $x\in \Lambda_t$.
\end{proof}

\begin{thm} \label{Tlevelper}
Let $u\in BV(\Omega)$. Then for a.e.\ $t\in \er$ we have
\begin{equation} \label{Elevel}
co\H^1(\Lambda_t)\le C\, P(E^t, \Omega)
\end{equation}
where $C$ depends only
on $C_P$, $C_{\lambda}$ and $C_{\mu}$.
\end{thm}

\begin{proof}
We first assume that $\Omega$ is bounded.
Let $T$ be the essential infimum of $u$. Then \eqref{Elevel}
obviously holds for $t<T$. If $t>T$, then $\mu(E^t)>0$
and then also $P(E^t, \Omega)>0$ by the isoperimetric inequality
(see e.g.\ \cite{KL}).
Denote $\psi(t)=P(E^t, \Omega)$
and note that $\psi$ is integrable, see \cite{A2} and
\cite{Mi}.
Let $\tau>T$ be a Lebesgue point for $\psi$ such that 
$\mu(\Lambda_\tau)=0$. By Lemma \ref{Lorder} and the Lebesgue differentiation theorem,
a.e.\ $\tau>T$ has these properties.
We show that $t=\tau$
has the required property. 
Choose $\delta>0$.  Lemma \ref{Lnull} gives $\alpha > 0$ such that for each bounded
open set  
$G$ with $\mu(G)< \alpha$
and $E \subset G$ we have
$$co\H^1_{\delta}(E) \leq C\, \Cp(E,G).$$

Now, using Lemma \ref{Lorder} we find
$a,b\in\er$ such that $a<\tau<b$,
$\psi(a)\le 2\psi(\tau)$, $\psi(b)\le 2\psi(\tau)$,
$\mu(\Lambda_a)=\mu(\Lambda_b)=0$ 
and $\mu(E_a^b)<\alpha$. We find an open
set $G \supset E_a^b$ such that still
$\mu(G) < \alpha$. 
Choose $x\in \Lambda_{\tau}$. Then $a<u^+(x)$, $u^-(x)<b$, and thus
$x\in \partial_*E_a$ (if $\overline D(E^a,x)>0$), or 
$x\in \partial_*E^b$ (if $\overline D(E_b,x)>0$), or 
$x\in \inte E_a^b$ (if $D(E^a,x)=D(E_b,x)=0$).
Summarizing,
$$
\Lambda_{\tau}\subset \partial_*E_a\cup \partial_*E^b\cup \inte E_a^b.
$$
We have
$$
\aligned
co\H^1_{\delta}(\partial_*E_a)&\le C\,P(E_a,\Omega)=C\,P(E^a,\Omega)\le 2C \, P(E^{\tau},\Omega),
\\
co\H^1_{\delta}(\partial_*E^b)&\le C \, P(E^b,\Omega))\le 2CP(E^{\tau},\Omega)
\endaligned
$$
and then
$$
\aligned
co\H^1_{\delta}(\inte E_a^b)&\le C\, \Cap(\inte E_a^b,G) \le  C\,P(E_a^b,G)
\\&
\le C(P(E^a,\Omega)+P(E^b, \Omega))\le 4C \, P(E^{\tau}, \Omega).
\endaligned
$$
Letting $\delta \rightarrow 0$ we obtain (\ref{Elevel}).

Suppose that $\Omega$ is unbounded. Fix a point $x_0 \in X$ and for each $i = 1,\,2, \, ...$ let 
$\Omega_i = \Omega \cap B(x_0, i)$ and $u_i = u|\Omega_i$. Denote by $E^t(u_i)$ the set $E^t$ associated
with $u_i$ and other sets, like $\Lambda_{\tau}(u_i)$, similarly. Now for a.e. $t \in \er$,
$\mu(\Lambda_t(u_i))=0$ for every $i$ and so for a.e. $t \in \er$
$$co\H^1(\Lambda_{t}(u_i)) \leq C\, P(E^t(u_i), \Omega_i) \leq C \,P(E^t(u_i), \Omega)
\le C\, P(E^{t}, \Omega) $$
for every $i$
and this easily implies (\ref{Elevel}) for $u$.
\end{proof}

If $u \in BV(\Omega)$ then by the co--area formula \cite[Proposition 4.2]{Mi} for the perimeter 
$P(E^t, \Omega) < \infty$ for
a.e. $t \in \er$. Hence Theorem \ref{Tlevelper} and Lemma \ref{Lorder} yield

\begin{cor} \label{Clevel}
If $u \in BV(\Omega)$, then 
$$
co\H^1(\Lambda_t)<\infty\text{ for a.e. }t\in\er.
$$
If, in addition,
$u$ is (approximately) continuous, then $co\H^1(u^{-1}(t)) < \infty$ for a.e. $t \in \er$.
\end{cor}

\begin{examples} \label{Exlevel}
The above corollary can be used to construct sets in $X$ whose boundaries have finite $co\H^1$--measure.
For example, let  $u \in BV(\Omega) \cap C(\Omega)$. Then for a.e. $t \in \er$ the boundary
of the open set $\{u > t\}$ has finite $co\H^1$--measure. For a more
specific example let $x_0 \in X$ and take $u(x) = d(x, x_0)$.
It follows that 
the topological boundary $\partial B(x_0,r)$ of the ball $B(x_0,r)$ has 
finite $co\H^1$--measure for a.e. $r > 0$. 
This is an improvement of the earlier results since it has been only known that $\mu(\partial B(x_0,r)) = 0$
except for a countable number of $r$ and that $co\H^1(\partial_* B(x_0,r)) < \infty$ for
a.e. $r>0$. More generally, if $K \subset X$ is a bounded set, then $u(x) = \operatorname{dist}(x, K)$ is a Lipschitz function and thus the boundary of 
the $t$--inflation $\{x :\, \dist(x,K) < t\}$  of $K$ has
finite $co\H^1$--measure for a.e. $t>0$.
\end{examples}


\end{document}